\documentclass[a4paper,11pt]{article}
\usepackage
[
colorlinks=true,
linkcolor=blue,
anchorcolor=blue,
citecolor=blue,
urlcolor=blue,
plainpages=false,
pdfpagelabels
]{hyperref}
\usepackage[english]{babel}
\usepackage[utf8]{inputenc}
\usepackage{amsmath,amssymb,amsthm}
\usepackage{txfonts}
\usepackage{graphicx}
\usepackage[colorinlistoftodos,textsize=scriptsize]{todonotes}
\usepackage{fullpage}
\usepackage[compact]{titlesec}
\usepackage[margin=0.5in]{caption}
\usepackage{url, inconsolata}
\usepackage{tikz, tikz-cd}
\usetikzlibrary{calc}
\usetikzlibrary{decorations.pathreplacing}
\usetikzlibrary{matrix,arrows,shapes,backgrounds,fit,positioning,calc,decorations.markings}
\tikzcdset{
arrow style=tikz,
diagrams={>={Straight Barb[scale=0.8]}}
}
\usepackage[capitalize]{cleveref}
\usepackage{xspace}
\usepackage{enumerate}
\usepackage[numbers,sort,compress]{natbib}
\usepackage{mathtools}

\newcommand\cat[1]{\ensuremath{\mathbf{#1}}}
\DeclareMathOperator{\im}{im}
\DeclareMathOperator{\coim}{coim}

\DeclareMathOperator{\coker}{coker}

\DeclareMathOperator{\rank}{rank}

\DeclareMathOperator{\id}{Id}

\DeclareMathOperator{\barccomp}{\bullet}

\newcommand{\barcEpi}{D^\twoheadrightarrow}
\newcommand{\barcMono}{D^{\hookrightarrow}}
\DeclareMathOperator{\card}{card}
\DeclareMathOperator{\supp}{supp}
\newcommand{\barcToDgm}{E}
\newcommand{\dgmToBarc}{F}

\newcommand{\kVect}[0]{\cat{Vect}}
\newcommand{\kvect}[0]{\cat{vect}}

\newcommand{\R}[0]{\mathbb R}

\newcommand{\IndMat}[0]{\mathcal X}

\newcommand{\A}[0]{\mathbf A}
\newcommand{\B}[0]{\mathcal B}
\newcommand{\C}[0]{\mathcal C}
\newcommand{\D}[0]{\mathcal D}

\newcommand{\I}[0]{\mathcal I}

\newcommand{\K}[0]{K}

\newcommand{\MatchingCat}[0]{\cat{Mch}}

\newcommand{\RCat}[0]{\cat{R}}

\newcommand{\Barc}[0]{\cat{Barc}}

\newcommand{\ShiftMap}[2]{S^{#1,#2}}

\newcommand{\pfd}{p.f.d.\@\xspace}

\usetikzlibrary{matrix,arrows,positioning,calc,decorations.markings}
\tikzset{strike thru/.style={
    decoration={markings, mark=at position 0.5 with {
        \draw [-] 
            ++ ( 0pt,-1.25pt) 
            -- ( 0pt, 1.25pt);}
    },
    postaction={decorate},
}}

\makeatletter
\newcommand*{\matching}{%
  \mathrel{%
    \mathpalette\@vneq{\rightarrow}%
  }%
}
\newcommand*{\@vneq}[2]{%
  \sbox0{\raisebox{\depth}{$#1\neq$}}%
  \sbox2{\raisebox{\depth}{$#1|\,\m@th$}}%
  \ifdim\ht2>\ht0 %
    \sbox2{\resizebox{\vneqxscale\width}{\vneqyscale\ht0}{\unhbox2}}%
  \fi
  \sbox2{$\m@th#1\vcenter{\copy2}$}%
  \ooalign{%
    \hfil\phantom{\copy2}\hfil\cr
    \hfil$#1#2\m@th$\hfil\cr
    \hfil\copy2\hfil\cr
  }%
}
\newcommand*{\vneqxscale}{1}
\newcommand*{\vneqyscale}{.3}
\makeatother

\renewcommand*{\matching}{
\to
}

\newtheorem{theorem}{Theorem}
\numberwithin{theorem}{section}
\newtheorem{proposition}[theorem]{Proposition}
\newtheorem{lemma}[theorem]{Lemma} 
\newtheorem{corollary}[theorem]{Corollary}

\theoremstyle{definition}
\newtheorem{example}[theorem]{Example}
\newtheorem{remark}[theorem]{Remark}

\newtheorem{definition}[theorem]{Definition}

\definecolor{myblue}{rgb}{0.22,0.45,0.70}
\definecolor{myred}{RGB}{228,36,20}
\definecolor{mygreen}{RGB}{40,168,55}
\definecolor{myyellow}{RGB}{255,204,4}
\definecolor{mypurple}{RGB}{160,80,154}
\definecolor{mygray}{gray}{0.85}
\definecolor{mylightgray}{gray}{0.9}

\title{Persistence Diagrams as Diagrams: \\ A Categorification of the Stability Theorem}
\author{
Ulrich Bauer%
\thanks{Technical University of Munich, Germany.
\url{mail@ulrich-bauer.org}}
\and Michael Lesnick%
\thanks{SUNY Albany, USA.
\url{mlesnick@albany.edu}}
}
\begin{document}
\maketitle
\begin{abstract}
Persistent homology, a central tool of topological data analysis, provides invariants of data called barcodes (also known as persistence diagrams).  A barcode is simply a multiset of intervals on the real line.  Recent work of Edelsbrunner, Jabłoński, and Mrozek suggests an equivalent description of barcodes as functors $\RCat\to \MatchingCat$, where $\RCat$ is the poset category of real numbers and $\MatchingCat$ is the category whose objects are sets and whose morphisms are matchings (i.e., partial injective functions).
Such functors form a category $\MatchingCat^\RCat$ whose morphisms are the natural transformations.  Thus, this interpretation of barcodes gives us a hitherto unstudied categorical structure on barcodes.  We show that this categorical structure leads to surprisingly simple reformulations of both the well-known stability theorem for persistent homology and a recent generalization called the induced matching theorem.  These reformulations make clear for the first time that both of these results can be understood as the preservation of certain categorical structure.  We also show that this perspective leads to a more systematic variant of the proof of the induced matching theorem.
\end{abstract}

\section{Introduction}\label{Sec:Intro}
The stability theorem for persistent homology is one of the main results of topological data analysis (TDA).  It plays a key role in the statistical foundations of TDA \cite{balakrishnan2013statistical}, and is used to formulate theoretical guarantees for efficient algorithms to approximately compute persistent homology \cite{Sheehy2013Linear,cavanna2015geometric}.  The theorem is originally due to Cohen-Steiner et al., who presented a version of the theorem for the persistent homology of $\R$-valued functions \cite{cohen2007stability}.  Since then, the theorem has been revisited a number of times, leading to simpler proofs and more general formulations \cite{chazal2009proximity,chazal2012structure,Bubenik2014Categorification,bauer2015induced,lesnick2015theory,Botnan2018Algebraic,bjerkevik2016stability,Bubenik2015Generalized}. In particular, Chazal et al.\@ introduced the \emph{algebraic stability theorem} \cite{chazal2009proximity}, a useful and elegant algebraic generalization, and it was later observed that the (easy) converse to this result also holds \cite{lesnick2015theory}. Bubenik and Scott were the first to explore the category-theoretic aspects of the stability theorem, rephrasing some of the key definitions in terms of functors and natural transformations \cite{Bubenik2014Categorification}.

Letting $\kvect$ denote the category of finite dimensional vector spaces over a fixed field~$\K$, a \emph{pointwise finite dimensional (\pfd) persistence module} is an object of the functor category $\kvect^\RCat$.  The structure theorem for \pfd
persistence modules \cite{crawley2015decomposition} tells us that the isomorphism type of a \pfd persistence module
$M$ is completely described by a unique collection of intervals called the \emph{barcode} $\B(M)$.
This barcode specifies how $M$ decomposes into indecomposable summands; such a decomposition is essentially unique.  The algebraic stability theorem, together with its converse, tells us that two persistence modules are algebraically similar (in a sense made precise by the language of \emph{interleavings}) if and only if they have similar barcodes.  

In \cite{bauer2015induced}, the authors of the present paper introduced the \emph{induced matching theorem}, an extension of the algebraic stability theorem to a general result about morphisms
of persistence modules, with a new, more direct proof.  The present paper is intended as a follow-up to \cite{bauer2015induced}.  The induced matching theorem can be viewed as a categorification of the stability theorem, and while this viewpoint was already present in \cite{bauer2015induced}, it was not fully developed.  Our goal here is to complete the development of the categorical viewpoint on induced matchings and algebraic stability.  %
 In order to make this paper self-contained, we revisit some of the same territory as \cite{bauer2015induced} along the way, leveraging the categorical perspective to streamline the presentation.

To formulate and prove the induced matching theorem, in \cite{bauer2015induced} we considered the category whose objects are barcodes and whose morphisms are arbitrary matchings (i.e., partial injective functions).
In the present paper, we introduce a different category of barcodes, denoted by $\Barc$, for which the morphisms are only those matchings satisfying a certain simple condition on how the matched intervals overlap.
We observe that there exists an equivalence of categories $\barcToDgm:\Barc\to\MatchingCat^\RCat$ extending the correspondence between barcodes and functors $\RCat\to \MatchingCat$ given by Edelsbrunner, Jab{\l}onski, and Mrozek \cite{edelsbrunner2013persistent}.  We use the category $\Barc$ to further develop the categorical viewpoint on stability.

Thanks to the equivalence $\barcToDgm$, it turns out all of the categorical structure of $\kvect^\RCat$ relevant to algebraic stability (as treated in \cite{bauer2015induced}) has an analogue in $\Barc$.  This allows us to present simple reformulations of both the induced matching and algebraic stability theorems, which make clear for the first time that both results can be understood as the preservation of certain categorical structure upon passing from persistence modules to barcodes.  Moreover, we show that this viewpoint leads naturally to a more systematic variant of the proof of the induced matching theorem (albeit one closely related to the proof given in \cite{bauer2015induced}).%

\subsection{Reformulation of the Induced Matching Theorem}
\label{sub:reformulationIMT}

\paragraph{Induced matchings}

To state the induced matching theorem, we need to first define a morphism of barcodes in $\Barc$ \[\IndMat(f):\B(M) \to \B(N)\]  induced by a morphism $f:M\to N$ of \pfd persistence modules.  This is called the \emph{induced matching of~$f$}.  To define $\IndMat(f)$, one first gives the definition in the case that $f$ is a monomorphism or epimorphism; see \cref{Sec: Induced_Matcings_Monos_Epis} for the details.  

For any category $\mathbf C$, let $\mathbf C^{\hookrightarrow}$ denote the subcategory with the same objects and morphisms the monomorphisms.  Similarly, let $\mathbf C^\twoheadrightarrow$ denote the subcategory with the same objects and morphisms the epimorphisms.  
The following result is equivalent to \cite[Proposition 4.2]{bauer2015induced}; we provide two different proofs, in \cref{Sec: Induced_Matcings_Monos_Epis,sec:rank_barcode}.
\begin{theorem}[Induced matchings for monos and epis]
\label{prop:SubmoduleStructureCat}
\mbox{}
\begin{enumerate}[(i)]
\item
The matchings induced by monomorphisms define a functor $\IndMat:(\kvect^\RCat)^{\hookrightarrow} \to \Barc^{\hookrightarrow}$.
\item 
Dually, the matchings induced by epimorphisms define a functor $\IndMat:(\kvect^\RCat)^\twoheadrightarrow \to \Barc^\twoheadrightarrow$.
\end{enumerate}
\end{theorem}

To extend the definition of the induced matchings $\IndMat(f)$ to arbitrary morphisms $f:M\to N$ of \pfd persistence modules, we take $\IndMat(f)=\IndMat(i)\circ \IndMat(q)$, where
\[M \stackrel{q}\twoheadrightarrow \im f\stackrel{i}\hookrightarrow N\]
is the epi-mono factorization of $f$.  Note that when $f$ is a monomorphism or epimorphism, this definition of coincides with the one given in \cref{prop:SubmoduleStructureCat} above.

\begin{remark}
The map $f\mapsto\protect \IndMat(f)$ is not functorial on all of $\kvect^\RCat$ \cite[Example 5.6]{bauer2015induced}, though it is functorial on the subcategories of monos and epis. Indeed, it is impossible to extend the map $M \mapsto \B(M)$ to a functor from $\kvect$ to $\Barc$ \cite[Proposition 5.10]{bauer2015induced}. 
\end{remark}

A morphism $f$ in $\kvect^\RCat$ is a monomorphism (epimorphism) if and only if $f$ has a trivial kernel (respectively, cokernel), and it can be checked that the same is true as well for a morphism $f$ in $\Barc$.  Thus, \cref{prop:SubmoduleStructureCat} tells us that the matchings induced by morphisms with trivial (co)kernels also have trivial (co)kernels. As formulated in this paper, the induced matching theorem is a generalization of this statement to small (but not necessarily trivial) (co)kernels.  

To make this precise, we need the following definition: 

\begin{definition}[$\delta$-trivial morphisms]
For $\A$ a pointed category (i.e., a category with a zero object) and $\delta\geq 0$, we say that a diagram $M:\RCat\to \A$ is \emph{$\delta$-trivial} if for all $t \in \R$, the internal morphism $M_{t,t+\delta}:M_t\to M_{t+\delta}$ is a zero morphism, i.e., it factors through the zero object.  The empty set is the zero object in $\MatchingCat$; we say a barcode $\C$ is $\delta$-trivial if $\barcToDgm(\C)$ is $\delta$-trivial.  
\end{definition}
Note that $M=0$ if and only if $M$ is $0$-trivial.  Using the definition of the equivalence $\barcToDgm$ given below in \cref{Subsec:BarcodesAsDiagrams}, it is straightforward to check that a barcode $\C$ is $\delta$-trivial if and only if each interval of $\C$ is contained in some half-open interval of length $\delta$.  
Moreover, a persistence module $M$ is $\delta$-trivial if and only if $\B(M)$ is $\delta$-trivial.

\begin{theorem}[Categorical formulation of the Induced Matching Theorem]\label{Thm:IMT_New_Formulation}
For any morphism $f:M\to N$ of \pfd persistence modules, the induced matching $\IndMat(f):\B(M)\to \B(N)$ is a morphism in $\Barc$ such that
\begin{enumerate}[(i)]
\item if $f$ has $\delta$-trivial kernel, then so does $\IndMat(f)$, and
\item if $f$ has $\delta$-trivial cokernel, then so does $\IndMat(f)$.
\end{enumerate}
\end{theorem}

Note that taking $\delta=0$ in \cref{Thm:IMT_New_Formulation}, we recover \cref{prop:SubmoduleStructureCat}.  In \cref{sec:IMT}, we give a concrete formulation of the induced matching theorem (\cref{Thm:InducedMatching}), similar to the version appearing in \cite{bauer2015induced}, and explain why the two formulations are equivalent.

\begin{remark}\label{Remark:Functoriality_Is_The_Heart}
In both the proof of the induced matching theorem given in \cite{bauer2015induced} and the proof given in the present paper, the first step is to prove \cref{prop:SubmoduleStructureCat}.  In this paper, we show that the proof of \cref{Thm:IMT_New_Formulation} follows readily from \cref{prop:SubmoduleStructureCat} and a simple characterization of the $\delta$-triviality condition for functors $\RCat\to \mathbf A$ taking values in a \emph{Puppe-exact category} $\mathbf A$; see \cref{Def:P_Exact,prop:imageSandwich}. 
\end{remark}
  
\begin{remark}
\cref{Thm:IMT_New_Formulation} has a simple converse, which we give in \cref{Prop:Conv_IMT}.
\end{remark}
  
\subsection{Reformulation of the Algebraic Stability Theorem}
We next turn to our reformulation of the algebraic stability theorem.  The theorem is typically formulated using the \emph{interleaving distance $d_I$} on persistence modules and the \emph{bottleneck distance} $d_B$ on barcodes; see \cref{Sec:Algebraic_Stability} for the definition.  Here, we use the categorical structure on barcodes to state the algebraic theorem purely in terms of interleavings of $\RCat$-indexed diagrams, 
without explicitly introducing $d_B$.  

Interleavings and the interleaving distance $d_I$ can be defined on $\RCat$-indexed diagrams taking values in an arbitrary category; see \cref{Def:Interleaving_Distance}.  By way of the equivalence $E$, we thus obtain definitions of interleavings and $d_I$ on $\Barc$; see \cref{SubSec:Interleavings_On_Barcodes}.
Our \cref{Prop:dI=db} establishes that the distances $d_I$ and $d_B$ on barcodes are equal; in fact, we give a slightly sharper statement.
From \cref{Prop:dI=db} it follows that the forward and converse algebraic stability theorems, as stated in \cite{bauer2015induced}, can be rephrased as follows:

\begin{theorem}[Categorical formulation of Algebraic Stability]\label{Thm:AST_New_Formulation}
Two \pfd persistence modules $M$ and $N$ are $\delta$-interleaved if and only if their barcodes $\B(M)$ and $\B(N)$ are $\delta$-interleaved.  In \mbox{particular,} \[d_I(M,N)=d_I(\B(M),\B(N)).\]
\end{theorem}

As we show in \cref{Sec:Algebraic_Stability}, this formulation of algebraic stability follows easily from \cref{Thm:IMT_New_Formulation}.

\subsection{Directly constructing barcodes and induced matchings of persistence modules as matching diagrams } 
In view of the equivalence $E:\Barc\to \MatchingCat^\RCat$, one may wonder whether one can give simple constructions of barcodes of persistence modules and induced matchings directly in the category $\MatchingCat^{\RCat}$.  In the final part of this paper, we explore this question.  Given a persistence module $M$, we give a direct construction of a matching diagram $\barcEpi(M)$ which is equivalent to the usual barcode of $M$.  $\barcEpi(M)$ is defined only in terms of the ranks of the linear maps in $M$; the definition does not depend on the structure theorem for persistence modules.  $\barcEpi(M)$ has the appealing property that the sets $\barcEpi(M)_r$ at each index $r$ are defined in an especially simple way, namely \[\barcEpi(M)_r=\{1,2,\ldots, \dim M_r\}.\]  We observe that, given an epimorphism of persistence modules $f:M\twoheadrightarrow N$, the matching induced by $f$ has a simple description as a natural transformation \[\barcEpi(f):\barcEpi(M) \twoheadrightarrow \barcEpi(N),\] and this leads to an alternate proof of \cref{prop:SubmoduleStructureCat}\,(ii).  There seems to be no comparably simple, direct description of the matching induced 
by a monomorphism $f:M \hookrightarrow N$ as a natural transformation $\barcEpi(M)\to \barcEpi(N)$.  But we observe that the matching diagram $\barcEpi(M)$ has a dual $\barcMono(M)$, also equivalent to the usual barcode, such that the matching induced by a monomorphism $f:M\hookrightarrow N$ has a simple description as a natural transformation \[\barcMono(f):\barcMono(M)\hookrightarrow \barcMono(N),\] leading (dually) to an alternate proof of \cref{prop:SubmoduleStructureCat}\,(i).

\subsection{Organization of the paper}
We begin \cref{Sec:Barcodes_As_Diagrams} by examining the properties of the category $\MatchingCat^\RCat$.  We then give the precise definitions of our category of barcodes $\Barc$ and of the equivalence $E:\Barc\to \MatchingCat^\RCat$.  As applications of this equivalence, we give a concrete description of (co)kernels and images in $\Barc$, and we describe how the $\delta$-triviality of the (co)kernel of a morphism $f:\mathcal C\to \mathcal D$ in $\Barc$ controls the similarity between $\mathcal C$ and $\mathcal D$.   In \cref{sec:IMT}, we use these descriptions to show that our categorical formulation of the induced matching theorem (\cref{Thm:IMT_New_Formulation}) is equivalent to a concrete formulation similar to that appearing in \cite{bauer2015induced}.  We then complete the definition of induced matchings and give our proof of the induced matching theorem.  In \cref{Sec:Interleavings_On_Barcodes}, we give the details of our reformulation of the algebraic stability theorem, and we prove that this follows easily from the induced matching theorem.  \cref{sec:rank_barcode} discusses the construction of barcodes and induced matchings directly in $\MatchingCat^{\RCat}$.

\section{Barcodes as diagrams}\label{Sec:Barcodes_As_Diagrams}

\subsection{Properties of $\MatchingCat$ and $\MatchingCat^\RCat$}
\label{Subsec:MatchingsPExact}
First, we review some basic properties of the category $\MatchingCat$ having sets as objects and matchings (partial injective functions) as morphisms. 
$\MatchingCat$ is a subcategory of the category  with sets as objects and relations as morphisms.
The composition $\tau\circ\sigma:S\matching U$ of two matchings $\sigma:S\matching T$ and $\tau:T\matching U$ is thus defined as
\begin{align*}
\tau\circ\sigma=\{(s,u) \mid (s,t) \in \sigma,\ (t,u) \in \tau \text{ for some } t \in T \}.
\end{align*}
The monomorphisms in $\MatchingCat$ are the injections, while the epimorphisms are the coinjections, i.e., matchings which match each element of the target. The kernel and cokernel of a morphism in $\MatchingCat$ consist of the unmatched elements of the source and target, respectively, together with the canonical (co)injections. Similarly, the image and coimage consist of the matched elements.

\begin{figure}
\centering
\begin{tikzpicture}[scale=0.5,baseline=-0.4ex,outer sep=0pt,inner sep=0pt]
\draw[thick,mygreen] (0,0) -- (2,0);
\draw[thick,mygreen] (0,-1) -- (2,-2);
\draw[thick,mygreen] (2,-1) -- (4,-1);
\draw[thick,mygreen] (2,-2) -- (4,-0);
\filldraw[myblue]
(0,0) circle (3pt) node(a){}
(0,-1) circle (3pt) node(b){}
(0,-2) circle (3pt) node(c){}
;
\filldraw[myred]
(2,0) circle (3pt) node(d){}
(2,-1) circle (3pt) node(e){}
(2,-2) circle (3pt) node(f){}
(2,-3) circle (3pt) node(g){}
(2,-4) circle (3pt) node(h){}
;
\filldraw[myyellow]
(4,0) circle (3pt) node(i){}
(4,-1) circle (3pt) node(j){}
(4,-2) circle (3pt) node(k){}
(4,-3) circle (3pt) node(l){}
;

\draw [->] (5.5,-1) -- (6.5,-1); 

\draw[thick,mygreen] (8,-1) -- (10,0);
\filldraw[myblue]
(8,0) circle (3pt) node(aa){}
(8,-1) circle (3pt) node(bb){}
(8,-2) circle (3pt) node(cc){};
\filldraw[myyellow]
(10,0) circle (3pt) node(hh){}
(10,-1) circle (3pt) node(ii){}
(10,-2) circle (3pt) node(jj){};

\begin{pgfonlayer}{background}
  \node (L) [fit=(a)(b)(c),style={rounded corners, inner sep=4pt, fill=myblue!30}] {};
  \node (M) [fit=(d)(e)(f)(g)(h),style={rounded corners, inner sep=4pt, fill=myred!30}] {};
  \node (R) [fit=(i)(j)(k)(l),style={rounded corners, inner sep=4pt, fill=myyellow!30}] {};

  \node (LL) [fit=(aa)(bb)(cc),style={rounded corners, inner sep=4pt, fill=myblue!30}] {};
  \node (RR) [fit=(hh)(ii)(jj),style={rounded corners, inner sep=4pt, fill=myyellow!30}] {};
\end{pgfonlayer}

\end{tikzpicture}
\qquad \qquad\qquad
\begin{tikzpicture}[scale=0.65,baseline=-0.4ex,outer sep=0pt,inner sep=0pt,label distance=2pt]
\filldraw[myblue]
(0,0) circle (3pt) node(a){}
(0,-1) circle (3pt) node(b){}
(0,-2) circle (3pt) node(c){};
\filldraw[myred]
(2,-1) circle (3pt) node(d){}
(2,-2) circle (3pt) node(e){}
(2,-3) circle (3pt) node(f){}
(2,-4) circle (3pt) node(g){};

\begin{pgfonlayer}{background}
  \node (coim) [fit=(b)(c),style={rounded corners, inner sep=4pt, fill=myblue!30},label=left:{\small$\coim f$}] {};
  \node (ker) [fit=(a),style={rounded corners, inner sep=4pt, fill=mygray},label=left:{\small$\ker f$}] {};
  \node(im) [fit=(d)(e),style={rounded corners, inner sep=4pt, fill=myred!30},label=right:{\small$\im f$}] {};
  \node(coker) [fit=(f)(g),style={rounded corners, inner sep=4pt, fill=mygray},label=right:{\small$\coker f$}] {};
\draw[thick,mygreen] (b) -- (d);
\draw[thick,mygreen] (c) -- (e);
\end{pgfonlayer}

\end{tikzpicture}
\caption{Examples illustrating matchings as a category. 
Left: the composition of two matchings. 
Right: kernel, coimage, image, and cokernel of a matching.}
\label{fig:matchings}
\end{figure}

\paragraph{$\MatchingCat$ and $\MatchingCat^\RCat$ as Puppe-exact categories}  The category $\MatchingCat$ is not Abelian: it does not have all binary (co)products, and is not even pre-additive.  Nevertheless, $\MatchingCat$ does share some structural similarities with an Abelian category.  In specific, $\MatchingCat$ is a \emph{Puppe-exact category}:

\begin{definition}\label{Def:P_Exact}
A Puppe-exact category \cite{grandis2012homological,Puppe1962Korrespondenzen,mitchell1965theory} 
is a category with the following properties: 
\begin{enumerate}
\item it has a zero object,
\item it has all kernels and cokernels,
\item every monomorphism is a kernel, and every epimorphism is a cokernel,
\item every morphism $f$ has an epi-mono factorization.
\end{enumerate}
\end{definition}

Every Abelian category is Puppe-exact, and it has been shown in \cite{grandis2012homological} that significant portions of homological algebra can be developed for Puppe-exact categories.

It follows from the definition that a Puppe-exact category also has all (co)images.  
Just like in Abelian categories, we have that \[\im f = \ker \coker f,\quad \coim f = \coker \ker f,\] and the coimage is canonically isomorphic to the image.  Moreover, the epi-mono factorization of a morphism $f$ is through $\im f$, and is essentially unique.  

For any category $\mathbf C$ and Puppe-exact category $\mathbf A$, the category of functors $\mathbf C\to \mathbf A$ is also Puppe-exact.  
Thus, $\MatchingCat^\RCat$ is Puppe-exact.  In particular, it has all kernels, cokernels, and images, and these are given pointwise.

\subsection{Barcodes}\label{Sec:Barcodes}

\begin{definition}[Multiset representations]
We say a \emph{multiset representation} is a subset $T\subseteq S\times X$ of sets $S$ and $X$, called the \emph{base set} and the \emph{indexing set} respectively.
For $s\in S$, the \emph{multiplicity} of $s$ in $T$ is the cardinality of the \emph{local indexing set} $X_s = \{x\in X\mid (s,x)\in T\}$.  In \cite{bauer2015induced}, we considered a more restrictive definition of a multiset representation, where the indexing set $X$ is $\mathbb N=\{1,2,3,\ldots\}$ and each local indexing set $X_s$ is required to be a prefix of $\mathbb N$; we refer to this as a \emph{natural multiset representation}.  (Using the more general definition here allows us to establish the link between barcodes and matching diagrams without imposing any cardinality conditions on the matching diagrams.)
\end{definition}

Let $T$ and $T'$ be multiset representations with the same indexing set $S$ and respective base sets $X$ and $X'$.  We say $T'$ \emph{reindexes} $T$, and write $T\cong T'$,
if there exists a bijection $f:T\to T'$ such that for all  $(s,x)\in T$, $f(s,x)=(s,x')$ for some $x'\in X'$.  Note that $\cong$ is an equivalence relation on multiset representations.

\begin{definition}[Barcode] An \emph{interval} in $\R$ is a non-empty set $I\subset \R$ such that if  $a,c\in I$ and $a<b<c$, then $b\in I$.  A \emph{barcode} is a multiset representation whose base set consists of intervals in $\R$.   If the barcode is a natural multiset representation, we call it a \emph{natural barcode}.  
\end{definition}
In working with barcodes, we often abuse notation slightly by suppressing the indexing set, and write an element $(s,x)$ of a barcode simply as $s$. 

\paragraph{Barcodes of Persistence Modules}
For $I$ an interval, define the \emph{interval module} $\K^I$ to be the persistence module such that 
\begin{align*}
\K^I_r&=
\begin{cases}
\K &{\textup{if }} r\in I, \\
0 &{\textup{ otherwise}.}
\end{cases}
& \K^I_{r,s}=
\begin{cases}
\id_\K &{\textup{if }} r,s\in I,\\
0 &{\textup{ otherwise}.}
\end{cases}
\end{align*}

The following well-known theorem tells us that natural barcodes arise as complete isomorphism invariants of \pfd persistence modules.

\begin{theorem}[Structure of \pfd persistence modules \cite{crawley2015decomposition}]
\label{thm:decompositionPFD}
For any \pfd persistence module $M$, there exists a unique natural barcode $\B(M)$ such that %
\[M\cong \bigoplus_{I\in \B(M)} \K^I.\] 
\end{theorem}
Following \cite{chazal2012structure}, we call this barcode $\B(M)$ the \emph{decomposition barcode} of $M$, or simply the barcode of $M$.

\subsection{The category of barcodes}
For intervals $I,J\subseteq \R$, we say that \emph{$I$ bounds $J$ above}
if for all $s \in J$ there exists $t \in I$ with $s \leq t$. 
If additionally $J$ bounds $I$ above, we say that $I$ and $J$ \emph{coincide above}.
Symmetrically, we say that \emph{$J$ bounds $I$ below}
if for all $t \in I$ there exists $s \in J$ with $s \leq t$, and that $I$ and $J$ \emph{coincide below} if additionally $I$ bounds $J$ below.
We say that
\emph{$I$ overlaps $J$ above} (and symmetrically, \emph{$J$ overlaps $I$ below}) if each of the following three conditions hold:
\begin{itemize}
\item $I\cap J\ne \emptyset$,
\item $I$ bounds $J$ above, and
\item $J$ bounds $I$ below.
\end{itemize}
For example, $[1,3)$ overlaps $[0,2)$ above, but neither $[0,4)$ nor $[0,2)$ overlap $[1,3)$ above. 
\begin{figure}
\centering

\begin{tikzpicture}[scale=.5]

\fill[mygray] (3,0) rectangle (6,-1);
\fill[mygray] (2,-1) rectangle (5,-2);
\fill[lightgray] (3,0) rectangle (5,-2);
\draw[ultra thick,myblue] (3,0) -- (9,0);
\node[myblue] at (9,0) [right] {$I$};
\draw[ultra thick,myred] (2,-1) -- (6,-1);
\node[myred] at (6,-1) [right] {$J$};
\draw[ultra thick,mygreen] (0,-2) -- (5,-2);
\node[mygreen] at (5,-2) [right] {$K$};

\end{tikzpicture}
\qquad
\begin{tikzpicture}[scale=.5]

\fill[mygray] (5,0) rectangle (6,-1);
\fill[mygray] (2,-1) rectangle (3,-2);
\draw[ultra thick,myblue] (5,0) -- (9,0);
\node[myblue] at (9,0) [right] {$I$};
\draw[ultra thick,myred] (2,-1) -- (6,-1);
\node[myred] at (6,-1) [right] {$J$};
\draw[ultra thick,mygreen] (0,-2) -- (3,-2);
\node[mygreen] at (3,-2) [right] {$K$};

\end{tikzpicture}
\caption{Illustration of overlap matchings and their composition.  Both the left and right examples depict overlap matchings $\sigma:\B\matching \C$ and $\tau:\C\matching \D$
between single-interval barcodes $\B = \{I\},$  $\C = \{J\},$ $\D  = \{K\}$, with $\sigma = \{(I,J)\},$ $\tau = \{(J,K)\}$.
We have $(I,K) \in \tau \barccomp \sigma$ if and only if $I\cap K \neq \emptyset$, so $\tau \barccomp \sigma=\{(I,K)\}$ for the left example, but $\tau \barccomp \sigma=\emptyset$ for the right example.
}
\label{fig:overlap_matchings}
\end{figure}

\begin{definition}[The category of barcodes]
We define an \emph{overlap matching} between barcodes $\C$ and $\D$ to be a matching 
 $\sigma:\C\matching \D$ such that if $\sigma(I)=J$, then $I$ overlaps $J$ above.  
Note that if $\sigma:\B\matching \C$ and $\tau:\C\matching \D$ are both overlap matchings, then the composition $\tau\circ \sigma$ in $\MatchingCat$ is not necessarily an overlap matching; for intervals $I,J,K$ such that $I$ overlaps $J$ above, and $J$ overlaps $K$ above, it may be that $I\cap K=\emptyset$, so that $I$ does not overlap $K$ above.

We thus define the \emph{overlap composition} $\tau \barccomp \sigma$ of overlap matchings $\sigma$ and $\tau$ as the matching
\begin{align*}
\tau\barccomp \sigma=\{(I,K)\in \tau\circ\sigma\mid \text{$I$ overlaps $K$ above.}\}
\end{align*}
See \cref{fig:overlap_matchings} for an illustration.  It is easy to check that with this new definition of composition, the barcodes and overlap matchings form a category, which we denote as $\mathbf{Barc}$.  
\end{definition}
Note that two barcodes are isomorphic in $\Barc$ if and only if one~reindexes the other.  Note also that the empty barcode is the zero object in $\Barc$.

\subsection{Barcodes as diagrams}
\label{Subsec:BarcodesAsDiagrams}

\paragraph{Functor from barcodes to diagrams}
We now define the equivalence $\barcToDgm:\Barc\to \MatchingCat^\RCat$.  For $\mathcal D$~a barcode and~$t\in \R$, we let 
\[\barcToDgm(\mathcal D)_t := \{I \in \mathcal D \mid t \in I\},\] and for each~$s\leq t$ we define the internal matching $\barcToDgm(\mathcal D)_{s,t}: \barcToDgm(\mathcal D)_s \matching \barcToDgm(\mathcal D)_t$ 
to be the restriction of the diagonal of $\mathcal D \times \mathcal D$ to $\barcToDgm(\mathcal D)_s \cap \barcToDgm(\mathcal D)_t$, i.e., \[\barcToDgm(\mathcal D)_{s,t}:=\{(I,I) \mid I \in \mathcal D, \; s,t \in I\}.\]  See \cref{fig:barcode_matchings} for an illustration.  

We define the action of $\barcToDgm$ on morphisms in $\Barc$ in the obvious way: for $\sigma:\C\matching \D$ an overlap matching and~$t\in \R$, we let $\barcToDgm(\sigma)_t: \barcToDgm(\C)_t \matching \barcToDgm(\mathcal D)_t$ be the restriction of $\sigma$ to pairs of intervals both containing~$t$, i.e.,
\[\barcToDgm(\sigma)_t:=\{(I,J) \in \sigma \mid t \in I \cap J \}. \]
 It is straightforward to check that $\barcToDgm$ is indeed a functor.
\begin{figure}
\centering
\begin{tikzpicture}[scale=1,x=4ex,y=5ex]

\draw[ultra thick,myblue] (0,0) -- (8,0);
\draw[ultra thick,gray] (1,-.25) -- (5,-.25);
\draw[ultra thick,myblue] (2,-.5) -- (6,-.5);
\draw[ultra thick,gray] (2,-.75) -- (4,-.75);
\draw[ultra thick,myblue] (4,-1) -- (10,-1);
\draw[ultra thick,myblue] (5,-1.25) -- (8,-1.25);

\draw[gray] (6,.5) -- (6,-1.75);

\node[gray] (t) at (6,-2.25) {$\barcToDgm(\C)_t$};

\filldraw[myblue]
(6,0) circle (2pt)
(6,-.5) circle (2pt)
(6,-1) circle (2pt)
(6,-1.25) circle (2pt);

\end{tikzpicture}
\qquad
\begin{tikzpicture}[scale=1,x=4ex,y=5ex]

\draw[ultra thick,myred] (0,0) -- (8,0);
\draw[ultra thick,gray] (1,-.25) -- (5,-.25);
\draw[ultra thick,myred] (2,-.5) -- (6,-.5);
\draw[ultra thick,gray] (2,-.75) -- (4,-.75);
\draw[ultra thick,gray] (4,-1) -- (10,-1);
\draw[ultra thick,gray] (5,-1.25) -- (8,-1.25);

\draw[gray] (3,.5) -- (3,-1.75);
\node[gray] (s) at (3,-2.25) {$\barcToDgm(\C)_s$};

\draw[gray] (6,.5) -- (6,-1.75);
\node[gray] (t) at (6,-2.25) {$\barcToDgm(\C)_t$};

\filldraw[myred]
(3,0) circle (2pt)
(3,-.5) circle (2pt)

(6,0) circle (2pt)
(6,-.5) circle (2pt);

\filldraw[gray]
(3,-.25) circle (2pt)
(3,-.75) circle (2pt)
(6,-1) circle (2pt)
(6,-1.25) circle (2pt);

\end{tikzpicture}
\caption{Examples illustrating the matching diagram representation $\barcToDgm(\C)$ of a barcode $\C$.  
Left: The intervals of $\barcToDgm_t(\C)$ are shown in blue (left).  Right: The intervals of $\coim \barcToDgm(\C)_{s,t}=\im \barcToDgm(\C)_{s,t}$ are shown in red.
}
\label{fig:barcode_matchings}
\end{figure}

\paragraph{Functor from Diagrams to Barcodes}
To see that $\barcToDgm$ is an equivalence, we next define a functor $F:\MatchingCat^\RCat\to \Barc$ such that $\barcToDgm$ and $F$ are inverses (up to natural isomorphism).  

For $D: \RCat \to \MatchingCat$, let
\[
\mathcal \dgmToBarc(D) := \left(\bigcup_{t\in\R} {\{t\} \times D_t} \right) \bigg/ \sim
\]
where $(t,x) \sim (u,y )$ if and only if $(x,y) \in D_{t,u}$ or $(y,x) \in  D_{u,t}$.  The functoriality of $D$ implies that the projection onto the first coordinate $(t,x) \mapsto t$ necessarily maps each equivalence class $Q \in \mathcal \dgmToBarc(D)$ to an interval $\supp(Q) = \{t \mid (t,x) \in Q \}\subseteq \R$.  We thus may define the barcode $\dgmToBarc(D)$ by
\[\dgmToBarc(D):=\{(\supp(Q),Q) \mid Q\in \mathcal \dgmToBarc(D)\},\]
where we interpret the above expression as a multiset representation by taking the index of each interval $\supp(Q)$ to be the equivalence class $Q$.
We take the action of $F$ on morphisms to be the obvious one: for diagrams $C,D: \RCat \to \MatchingCat$ and  $\eta:C \to D$ a natural transformation (consisting of a family of matchings $\eta_t: C_t \matching D_t$), we take $\dgmToBarc(\eta): \dgmToBarc(C) \matching \dgmToBarc(D)$ to be the overlap matching given by
\[\dgmToBarc(\eta):=\left\{\left((\supp(Q),Q),(\supp(R),R)\right) \mid Q \in \mathcal \dgmToBarc(C),\  R \in \mathcal \dgmToBarc(D),  \exists\ t\in \R,\ (x,y) \in \eta_t : (t,x) \in Q, (t,y) \in R \right\}. \]
It is easy to check that $F$ is a functor and that $\barcToDgm$ and $F$ are indeed inverses up to natural isomorphism.

\subsection{Kernels, cokernels, and images of barcodes}
In the induced matching approach to algebraic stability, (co)kernels and $\delta$-triviality of persistence modules both play an essential role.  We have seen above that the definitions of these extend to functor categories $\mathbf A^\RCat$ for any Puppe-exact category $\mathbf A$; in particular, they extend to $\MatchingCat^\RCat$.  Thus, since $\MatchingCat$ is equivalent to $\Barc$, these definitions also carry over to $\Barc$. 

We next give concrete descriptions of kernels, cokernels, and images in $\Barc$.  We then use these to obtain a simple description of how the $\delta$-triviality of the (co)kernel of a morphism $f:\mathcal C\to \mathcal D$ in $\Barc$ controls the similarity between $\mathcal C$ and $\mathcal D$. %

For $\sigma:\C\matching \D$ an overlap matching of barcodes and $I\in \C$, define 
 \[\ker(\sigma,I)=
 \begin{cases} 
 I &\textup{if } \sigma \textup{ does not match }I,\\ 
 I \setminus J & \textup{if } \sigma(I)=J.
 \end{cases}\]
Hence, $\ker(\sigma,I)$ is either empty or an interval in $\R$. 
In the latter case, $I$ and  $\ker(\sigma,I)$ coincide above.
Dually, for $J\in \D$, we define 
 \[\coker(\sigma,J)=
 \begin{cases} 
 J &\textup{if } \sigma \textup{ does not match }J,\\ 
 J \setminus I & \textup{if } \sigma(I)=J.
 \end{cases}\]
\begin{proposition}\label{Prop:Kers_Cokers_of_Barcodes}
For any morphism (i.e., overlap matching) $\sigma:\C\matching \D$ in $\Barc$,
the categorical kernel, cokernel, and image of $\sigma$ exist and are given by 
\begin{align*}
 \ker \sigma&= \{\ker(\sigma,I) \neq \emptyset \mid {I\in \C}\}, \\
 \quad \coker \sigma&= \{\coker(\sigma,J) \neq \emptyset \mid {J\in \D}\},\\
 \im \sigma&= \{ I \cap J \mid {(I,J)\in \sigma}\}.
 \end{align*}
 \end{proposition}
 
\begin{proof}
Given $\sigma:\C\matching \D$ in $\Barc$, applying the equivalence $E$ yields a morphism of matching diagrams $E(\sigma)$ such that 
\[(\ker E(\sigma))_t=\{I\in \C\mid t\in I\textup{ and $I$ is not matched by $\sigma$ to an interval $J \in \D$ with $t \in J$}\}.\]
It is then clear that \[F(\ker E(\sigma))= \{\ker(\sigma,I) \neq \emptyset \mid {I\in \C}\}.\]
Since \[\ker \sigma \cong F\circ E(\ker \sigma)\cong F(\ker(E(\sigma))),\] the result for kernels holds.  
Similar arguments give the results for cokernels and for images.
\end{proof}

Using this concrete description of (co)kernels in $\Barc$, we now give an explicit description of the notion of $\delta$-triviality for (co)kernels of overlap matchings.  %
Given an interval $I \subset \R$ and $\delta\geq 0$, let
\begin{equation}\label{Eq:Interval_Shift}
I(\delta) := \{t \mid t +\delta \in I\}
\end{equation}
be the interval obtained by shifting $I$
downward by $\delta$.

\begin{proposition}\label{Prop:Triviality_Kers_Cokers_of_Matching_Diagram}
Let $\eta:\C\to \D$ an overlap matching of barcodes.
Then
\begin{enumerate}[(i)]
\item $\ker \eta$ is $\delta$-trivial if and only if
\begin{enumerate}
\item for each $(I,J) \in \eta$, $J$ bounds $I(\delta)$ above, and
\item any interval of $\C$ that is not matched by $\eta$ is contained in a half-open interval of length $\delta$.
\end{enumerate}
\item $\coker \eta$ is $\delta$-trivial if and only if
\begin{enumerate}
\item for each $(I,J) \in \eta$, $I(\delta)$ bounds $J$ below, and
\item any interval of $\D$ that is not matched by $\eta$ is contained in a half-open interval of length $\delta$.
\end{enumerate}
\end{enumerate}
\end{proposition}
\begin{proof}
As noted in \cref{sub:reformulationIMT}, a barcode $\mathcal C$ is $\delta$-trivial if and only if each interval in $\mathcal C$ is contained in a half-open interval of length $\delta$.  Given this, the result follows immediately from \cref{Prop:Kers_Cokers_of_Barcodes}.
\end{proof}

Recall that a morphism has $0$-trivial (co)kernel if and only if it is a monomorphism (epimorphism).  We thus have the following corollary of \cref{Prop:Triviality_Kers_Cokers_of_Matching_Diagram}, which gives a concrete interpretation of \cref{prop:SubmoduleStructureCat}:
\begin{corollary}%
\label{Thm:SubmoduleStructure}
Let $\eta:\C\to \D$ an overlap matching of barcodes. Then
\begin{enumerate}[(i)]
\item $\eta$ is a monomorphism if and only if
\begin{enumerate}
\item for each $(I,J) \in \eta$, $I$ and $J$ coincide above, and
\item every interval of $\C$ is matched (i.e., $\eta$ is an injection).
\end{enumerate}
\item $\eta$ is an epimorphism if and only if
\begin{enumerate}
\item for each $(I,J) \in \eta$, $I$ and $J$ coincide below, and
\item every interval of $\D$ is matched (i.e., $\eta$ is a coinjection).
\end{enumerate}
\end{enumerate}
\end{corollary}

\section{The induced matching theorem}
In this section, we observe that the categorical formulation of the induced matching theorem (\cref{Thm:IMT_New_Formulation}) is equivalent to a more concrete statement, similar to the formulation appearing in \cite{bauer2015induced}. We then define the matchings induced by epimorphisms and monomorphisms of persistence modules, thereby completing the definition of induced matchings given in \cref{sub:reformulationIMT}.   To finish the section, we prove the induced matching theorem, working directly with the categorical formulation of the theorem.

\label{sec:IMT}

\subsection{Concrete formulation of the induced matching theorem}

It follows from \cref{Prop:Kers_Cokers_of_Barcodes,Prop:Triviality_Kers_Cokers_of_Matching_Diagram} that our categorical reformulation of the induced matching theorem (\cref{Thm:IMT_New_Formulation}) is equivalent to the following.  %
 See \cref{fig:trivialKernel} for an illustration.

\begin{theorem}[Induced Matching Theorem \cite{bauer2015induced}]\label{Thm:InducedMatching}
Let $f:M\to N$ be a morphism of \pfd persistence modules.
\begin{enumerate}[(i)]
\item The induced matching $\IndMat(f):\B(M)\matching \B(N)$ is an overlap matching.
\item If $\ker f$ is $\delta$-trivial, then
\begin{enumerate}
\item for each $(I,J) \in \IndMat(f)$, $J$ bounds $I(\delta)$ above, and
\item any interval of $\B(M)$ not matched by $\IndMat(f)$ is contained in a half-open interval of length $\delta$.
\end{enumerate}
\item If $\coker f$ is $\delta$-trivial, then 
\begin{enumerate}
\item for each $(I,J) \in \IndMat(f)$, $I(\delta)$ bounds $J$ below, and
\item any interval of $\B(N)$ not matched by $\IndMat(f)$ is contained in a half-open interval of length $\delta$.
\end{enumerate}
\end{enumerate}
\end{theorem}

\begin{figure}
\centering
\begin{tikzpicture}[scale=.5,baseline=(bottom),x=1cm,y=1.25cm]
\draw[ultra thick,myyellow] (5,0) -- (7,0);
\draw[ultra thick,myyellow] (5.5,-0.25) -- (6,-0.25);
\node[myyellow] at (8,0) [right] {$\B(\ker f)$};
\draw[lightgray,dash pattern=on 0pt off 2pt,line cap=round,thick] (2,-1) -- (2,-2);
\draw[ultra thick,myblue] (2,-1) -- (7,-1);
\draw[ultra thick,myblue] (5,-1.25) -- (6,-1.25);
\node[myblue] at (8,-1) [right] {$\B(M)$};
\node[myblue] at (2,-1) [left] {$I$};
\draw[lightgray,dash pattern=on 0pt off 2pt,line cap=round,thick] (5.5,-2) -- (5.5,-3);
\draw[ultra thick,mygreen] (2,-2) -- (5.5,-2);
\node[mygreen] at (8,-2) [right] {$\B(\im f)$};
\draw[ultra thick,myred] (0.5,-3) -- (5.5,-3);
\node[myred] (bottom) at (8,-3) [right] {$\B(N)$};
\node[myred] (bottom) at (0.5,-3) [left] {$J$};
\draw[ultra thick,myblue] (0,-4) -- (5,-4);
\draw[ultra thick,myblue] (2.5,-4.25) -- (4,-4.25);
\node[myblue] (bottom) at (8,-4) [right] {$\B(M(\delta))$};
\node[myblue] at (-1,-4) [left] {$I(\delta)$};
\end{tikzpicture}
\caption{Illustration for part (ii) of the induced matching theorem: the right endpoint of the interval $J \in \B(N)$ coincides with that of an interval in $\B(\im f)$ and lies between the right endpoint of the interval $I \in \B(M)$ and that of the shifted interval $I(\delta) \in \B(M(\delta))$.}
\label{fig:trivialKernel}
\end{figure}

\subsection{Matchings induced by monos and epis of persistence modules}\label{Sec: Induced_Matcings_Monos_Epis}

We now define the matching $\IndMat(f)$ induced by a monomorphism or epimorphism $f$ of persistence modules.  The way we will present the definition will depend on a structural result, \cref{Prop:Ind_Matchings_for_Monos} below, which also leads almost immediately to a proof of \cref{prop:SubmoduleStructureCat}.

Let $\mathcal \I$ denote the set of intervals in $\R$.  For $I,J\in \mathcal I$, write $I\sim_a J$ if $I$ and $J$ coincide above.  $\sim_a $ is an equivalence relation on $\I$.  For $\B$ a barcode, $\sim_a$ induces an equivalence relation on $\B$, which we also denote as $\sim_a$.  For each equivalence class $e\in \mathcal I/{\sim_a}$, let $\B^e$ denote the corresponding equivalence class of $\B/{\sim_a}$ if $\B$ contains any intervals in $e$.  Otherwise let $\B^e=\emptyset$.  If $\B$ is the barcode of a \pfd module, then each $\B^e$ is finite or countable.  In addition, if $\B^e$ is non-empty then it contains a maximal interval under inclusion.  We endow $\B^e$ with a total order by taking $(I,n)<(J,n')$ if $I$ strictly contains $J$ or $I=J$ and $n<n'$.  $\B^e$ is then a countable, well-ordered set, hence isomorphic to a prefix of $\mathbb N$.

\begin{proposition}[Induced Matchings for Monos]\label{Prop:Ind_Matchings_for_Monos}
If $f:M\to N$ is a monomorphism of persistence modules, then 
\begin{enumerate}
\item[(i)] for each $e\in \I/{\sim_a}$, \[|\B(M)^e|\leq|\B(N)^e|.\]  Thus, we have a well defined injection $\IndMat(f):\B(M)  \hookrightarrow \B(N)$, which sends the $i^{th}$ element of $\B(M)^e$ to the $i^{th}$ element of $\B(N)^e$.
\item[(ii)] $\IndMat(f)$ is in fact a monomorphism in $\Barc$.
\end{enumerate}
\end{proposition}

A simple proof of \cref{Prop:Ind_Matchings_for_Monos} is given in \cite[Section 4]{bauer2015induced}.  Here, we present a variant of that argument.

\begin{proof}[Proof of \cref{Prop:Ind_Matchings_for_Monos}]
For any interval $I\subset \R$, we define a functor $F^I:\kvect^\RCat\to \kvect$ such that 
\begin{enumerate}
\item for all \pfd persistence modules $M$, $\dim F^I(M)$ is the number of intervals in $\B(M)$ which contain $I$ and coincide with $I$ above, and
\item $F^I$ maps monomorphisms to monomorphisms.
\end{enumerate}
To define $F^I$, we choose $t\in I$ and let 
\[\ker^+= \bigcap_{u\not \in I: t<u} \ker M_{t,u}, \quad  \ker^-= \bigcup_{u \in I: t \leq u} \ker M_{t,u}, \quad \im^+=\bigcap_{s\in I: s \leq t} \im M_{s,t}.\]
We take \[F^I(M)=(\ker^+\cap \im^+)/(\ker^-\cap \im^+).\]
The map $M\mapsto F^I(M)$ is easily checked to be functorial.
From the structure theorem \ref{thm:decompositionPFD}, it is clear that $\dim F^I(M)$ has the desired property, and it is straightforward to check that $F$ preserves monomorphisms.

The proposition follows easily from the existence of the functors $F^I$: Let $I$ be the $j^\mathrm{th}$ interval in $\B(M)^e$.  We have \[j\leq \dim F^I(M)\leq \dim F^I(N)\leq |\B(N)^e|.\]  If $\B(M)^e$ is finite, then taking $j=|\B(M)^e|$ gives that $|\B(M)^e|\leq|\B(N)^e|$.  If $\B(M)^e$ is countably infinite, then we have that $j\leq |\B(N)^e|$ for all $j\geq 0$, hence $|\B(N)^e|$ is infinite as well.  This proves (i).

To prove (ii), note that for each $I\in \B(M)$, $I$ and $\IndMat(f)(I)$ coincide above, so in view of \cref{Thm:SubmoduleStructure}, it suffices to show that $I\subset \IndMat(f)(I)$.  Suppose that $I$ is the $j^\mathrm{th}$ interval in $\B(M)^e$.  Since $\dim F^I(M)\leq \dim F^I(N)$, $\B(N)^e$ has at least $j$ intervals containing $I$.  $\IndMat(f)(I)$ is by definition the $j^\mathrm{th}$ interval of $\B(N)^e$, so we have $I\subset \IndMat(f)(I)$, as desired.
\end{proof}

To define $\IndMat(f)$ for an epimorphism $f$, we simply dualize the above construction, taking two intervals to be equivalent if and only if they coincide below.  The dual argument shows that $\IndMat(f)$ is an epimorphism in $\Barc$.

\begin{proof}[Proof of \cref{prop:SubmoduleStructureCat} (induced matchings for monos and epis)]
It is easy to see that the map $f\mapsto \IndMat(f)$ of the \cref{Prop:Ind_Matchings_for_Monos} is in fact functorial, so this defines a functor $\IndMat$ from monomorphisms of persistence modules to monomorphisms in $\Barc$, proving \cref{prop:SubmoduleStructureCat}\,(i).  The dual observation yields  \cref{prop:SubmoduleStructureCat}\,(ii).  
\end{proof}

\begin{example}
Interestingly, the map $f\mapsto\protect \IndMat(f)$ may strictly decrease the triviality of (co)kernels: we give an example of a monomorphism $f : M \hookrightarrow N$ such that $\coker f$ is not $2$-trivial but $\coker \IndMat(f)$ is $2$-trivial.  Let \[M = K^{[2,4)},\quad N= K^{[0,4)} \oplus K^{[1,3)},\quad \textup{and}\quad f = \begin{pmatrix} 1\\1 \end{pmatrix}.\] Then $\B(\coker f) = \{[0,3),[1,2)\}$ but $\coker \IndMat(f) = \{[0,2),[1,3)\}$.
In contrast, note that for any morphism $f$, we have by construction that $\im \IndMat(f) = \B(\im f)$.  
\end{example}

\subsection{A characterization of morphisms with $\delta$-trivial (co)kernel}\label{Sec:Characterization}
We now turn our attention to the proof of the induced matching theorem.  First, we introduce some notation.

\paragraph{Shifts of $\RCat$-indexed diagrams and barcodes}
Consider the translation $t \mapsto t + \delta$ of the real line by $\delta \in \R$ as an endofunctor $S_\delta: \RCat \to \RCat$.  For any category $\A$ and diagram $M: \RCat \to \A$, we write 
$M(\delta) := M \circ S_\delta$.  
Thus, $M(\delta)$ is the diagram obtained by shifting each vector space and linear map in $M$ downward by $\delta$.  Given $M,N:\RCat \to \A$, a morphism $f:M\to N$ induces a morphism $f(\delta):M(\delta)\to N(\delta)$.  

For $\delta\geq 0$, the internal morphisms $\{M_{t,t+\delta}\}_{t\in \R}$ assemble into a natural transformation $M\to M(\delta)$, which we denote by $\ShiftMap{M}{\delta}$.  Note that since $(M(-\delta))(\delta)=M$, we have a natural transformation $\ShiftMap{M(-\delta)}{\delta}:M(-\delta)\to M$.

For $\C$ a barcode, let \[\C(\delta):=\{I(\delta)\mid I\in \C\},\] where $I(\delta)$ is as defined in \cref{Eq:Interval_Shift}, and let $\ShiftMap{\C}{\delta}:\C\to \C(\delta)$ be the overlap matching given by 
\[\ShiftMap{\C}{\delta}:=\{(I,I(\delta))\mid I\textup{ is not }\delta\textup{-trivial}\}.\]  Note that for $E:\Barc\to \MatchingCat^\RCat$ the equivalence of \cref{Subsec:BarcodesAsDiagrams}, $E(\ShiftMap{\C}{\delta})=\ShiftMap{E(\C)}{\delta}$.%

The following proposition is one of the key ingredients in our proof of the induced matching theorem:

\label{Subsec:deltaTrivial}
\begin{lemma}
\label{prop:imageSandwich}
Given diagrams $M, N: \RCat \to \A$ with $\A$ Puppe-exact, and a morphism $f: M \to N$ 
with epi-mono factorization
\[M \stackrel{q}\twoheadrightarrow \im f \stackrel{i}\hookrightarrow N,\] 
the following are equivalent:
\begin{enumerate}[(i)]
\item $\ker f$ is $\delta$-trivial;

\item the image epimorphism $r:M \twoheadrightarrow \im \ShiftMap{M}{\delta}$ factors as
\[
\begin{tikzcd}
M \ar[r,two heads,"r"] \ar[d,two heads,"q",swap] & \im \ShiftMap{M}{\delta} \\
\im f \ar[ur,two heads,"p",swap,dashed]
\end{tikzcd}
\]
for some epimorphism $p: \im f \twoheadrightarrow \im \ShiftMap{M}{\delta}$.

\end{enumerate}
Dually, the following are equivalent:
\begin{enumerate}[(i)]
\item $\coker f$ is $\delta$-trivial;

\item the image monomorphism $h:\im (\ShiftMap{N(-\delta)}{\delta}) \hookrightarrow N$ factors as
\[
\begin{tikzcd}
\im \ShiftMap{N(-\delta)}{\delta} \ar[r,hook,"h"] \ar[dr,hook,"j",swap,dashed] &N \\
& \im f \ar[u,hook,"i",swap]
\end{tikzcd}
\]
for some monomorphism $j: \im \ShiftMap{N(-\delta)}{\delta} \hookrightarrow \im f$.

\end{enumerate}
\end{lemma}

\begin{proof}
We give the proof for $\ker f$, the dual case of $\coker f$ being analogous.  Let \[\kappa : \ker f \hookrightarrow M\quad \textup{and} \quad\mu: \ker {\ShiftMap{M}{\delta}} \hookrightarrow M\] denote the kernel monomorphisms, and let \[q : M \twoheadrightarrow \im f\quad \textup{and} \quad r: M \twoheadrightarrow \im {\ShiftMap{M}{\delta}}\] denote the image epimorphisms.

To show that (i) implies (ii), assume that  $\ker f$ is $\delta$-trivial, i.e.,
\[\ShiftMap{\ker f}{\delta}:\ker f \to \ker f(\delta)\] is the zero morphism.
Then we also have 
\[\ShiftMap{M}{\delta} \circ \kappa = \kappa(\delta) \circ \ShiftMap{\ker f}{\delta}=0.\]
The universal property of the kernel monomorphism $\mu$ thus provides a unique morphism
$v:\ker f \to \ker {\ShiftMap{M}{\delta}}$ such that $\kappa = \mu \circ v$. 
\[
\begin{tikzcd}
& \ker f \ar[r] \ar[dl,dashed,hook,swap,"\exists! v"] \ar[d,hook,"\kappa"] & 0 \ar[r]  \ar[d] & \ker f(\delta) \ar[d,hook,"\kappa(\delta)"] \\
\ker {\ShiftMap{M}{\delta}} \ar[r,hook,"\mu"] & M \ar[r,two heads,"r"] \ar[d,two heads,"q"] & \im {\ShiftMap{M}{\delta}} \ar[r,hook] & M(\delta) \\
& \im f \ar[ur,dashed,two heads,swap,"\exists!p"] 
\end{tikzcd}
\]
Since $\kappa$ is a monomorphism, $v$ must be a monomorphism too.  
We have $r \circ \mu=0$,
so \[ r \circ \kappa = r \circ \mu  \circ v=0\] as well.
Now by the universal property of $q$ as the cokernel epimorphism of $\kappa$,
there is a unique epimorphism $p: \im f \to \im {\ShiftMap{M}{\delta}}$ such that $r = p \circ q$.

To show that (ii) implies (i),
assume that there is an epimorphism $p$ factoring $r = p \circ q$.
We have $q \circ \kappa=0$, 
so \[r \circ \kappa = p \circ q \circ \kappa=0\] as well.
Thus \[\ShiftMap{M}{\delta} \circ \kappa = \kappa(\delta) \circ \ShiftMap{\ker f}{\delta}=0,\] and since $\kappa(\delta)$ is a monomorphism, this implies that $\ShiftMap{\ker f}{\delta}=0$.
\end{proof}

\subsection{Proof of the induced matching theorem}
To prove the induced matching theorem (\cref{Thm:IMT_New_Formulation}) we will need the following lemma, which follows easily from the definition of induced matchings and the structure theorem for persistence modules (\cref{thm:decompositionPFD}).
\begin{lemma}\label{Lem:Matching_Induced_By_Shifts}
For any \pfd persistence module $M$, we have $\ShiftMap{\B(M)}{\delta}=\IndMat(\ShiftMap{M}{\delta})$.
\end{lemma}

\begin{proof}[Proof of the induced matching theorem (\cref{Thm:IMT_New_Formulation})]
We prove (i); the proof dualizes to a proof of (ii).  
Write $s = \ShiftMap{M}{\delta}$, and
let $f=i\circ q$ and $s=j\circ r$ be the epi-mono factorizations.  By \cref{prop:imageSandwich}, 
we obtain an epimorphism $p: \im f \twoheadrightarrow \im s$ such that
the following diagram commutes:
\[
\begin{tikzcd}
M \ar[r,two heads,"r"]\ar[rr,bend left,"s"]\ar[dd,bend right=40,"f",swap] \ar[d,two heads,"q"] & \im s \ar[r,hook,"j"] & M(\delta) \\
\im f \ar[d,hook,"i"] \ar[ur,two heads,"p",swap,dashed]  \\
N
\end{tikzcd}
\]

By  \cref{prop:SubmoduleStructureCat} and the way we construct induced matchings, we have epi-mono factorizations $\IndMat(f)=\IndMat(i)\circ \IndMat(q)$ and $\IndMat(s) = \IndMat(j)\circ \IndMat(r)$.
Moreover,
$\IndMat$ is functorial on epimorphisms by \cref{prop:SubmoduleStructureCat}, so the following diagram also commutes:
\[
\begin{tikzcd}
\B(M) \ar[r,two heads,"\IndMat(r)"]\ar[rr,bend left,"\IndMat(s)"]\ar[dd,bend right=50,"\IndMat(f)",swap] \ar[d,two heads,"\IndMat(q)"] & \B(\im s) \ar[r,hook,"\IndMat(j)"] & \B(M(\delta)) \\
\B(\im f) \ar[d,hook,"\IndMat(i)"] \ar[ur,two heads,"\IndMat(p)",swap,dashed]  \\
\B(N)
\end{tikzcd}
\]
By \cref{Lem:Matching_Induced_By_Shifts}, we have $\ShiftMap{\B(M)}{\delta} = \IndMat(s)$.
Thus, since epi-mono factorizations are unique (up to unique isomorphism), we have
$
\im \ShiftMap{\B(M)}{\delta} =
\B(\im s),
$
and $\IndMat(r)$ is the image epimorphism $\B(M)\twoheadrightarrow \im \ShiftMap{\B(M)}{\delta}$.  Since $\IndMat(r)=\IndMat(q)\circ \IndMat(p)$,  \cref{prop:imageSandwich} now gives that $\ker \IndMat(f)$ is $\delta$-trivial, as desired.
\end{proof}

\subsection{Converse to the Induced Matching Theorem}\label{Sec:Converse_IMT}
Letting $\kVect$ denote the category of (not necessarily finite dimensional) vector spaces over the field~$\K$.  We have a functor $\mathrm{Fr}:\MatchingCat\to \kVect$, which takes a set $S$ to the vector space with basis $S$.  Let $\zeta:\Barc\to \kVect$  denote the functor which sends a barcode $\C$ to $\mathrm{Fr}\circ E(\C)$.  

It is easy to prove the following converse to the induced matching theorem:
\begin{proposition}\label{Prop:Conv_IMT}
\mbox{}
\begin{enumerate}[(i)]
\item $\zeta(\B(M))\cong M$ for any \pfd persistence module $M$.
\item If $f:\mathcal C \to \mathcal D$ is a morphism in $\Barc$ with $\delta$-trivial (co)kernel, then $\zeta(f)$ has $\delta$-trivial (co)kernel as well.
\end{enumerate}
\end{proposition}

\section{Interleavings of barcodes and the bottleneck distance}\label{Sec:Interleavings_On_Barcodes}
In this section, we consider interleavings and the bottleneck distance on barcodes.  We observe that the bottleneck distance can be interpreted as an interleaving distance, and we prove the algebraic stability theorem.

\subsection{Interleavings}
\paragraph{Interleavings of $\RCat$-indexed diagrams}\label{SubSec:Interleavings_On_Barcodes}

The definition of interleavings of $\RCat$-indexed diagrams was introduced in \cite{chazal2009proximity}, building on ideas in \cite{cohen2007stability}, and was first stated in categorical language in \cite{Bubenik2014Categorification}.  Though interleavings over more general indexing categories can be defined and are also of interest in TDA \cite{lesnick2015theory,Bubenik2015Generalized,de2016categorified,curry2014sheaves}, we focus here on the $\RCat$-indexed case.  We use the definitions and notation introduced in \cref{Sec:Characterization}.

\begin{definition}[Interleavings and interleaving distance]\label{Def:Interleaving_Distance}
A $\delta$-\emph{interleaving} between two diagrams $M, N: \RCat \to \A$ is a pair of natural transformations \[f: M \to N(\delta),\quad g: N \to M(\delta)\] such that $g(\delta) \circ f = \ShiftMap{M}{2\delta}$ and $f(\delta) \circ g = \ShiftMap{N}{2\delta}$.  We call $f$ and $g$ \emph{$\delta$-interleaving morphisms}.

The interleaving distance on objects of $\A^\RCat$ is then given by 
\[d_I(M,N):=\inf\, \{\delta\geq 0\mid M\textup{ and }N\textup{ are }\delta\textup{-interleaved}.\}\]
\end{definition}

\paragraph{Interleavings in $\Barc$}
Note that as for natural transformations of $\RCat$-indexed diagrams, an overlap matching $f:\C\to \D$ induces an overlap matching $f(\delta):\C(\delta)\to \D(\delta)$.  We define a $\delta$-interleaving between barcodes $\C$ and $\D$ to be a pair of overlap matchings \[f:\C\matching  \D(\delta),\qquad g:\D\matching  \C(\delta)\]
such that 
$g(\delta)\barccomp f=\ShiftMap{\C}{2\delta}$, and $f(\delta)\barccomp g=\ShiftMap{\D}{2\delta}$.  This definition is equivalent to the definition of interleavings in $\MatchingCat^\RCat$ in the sense that a pair of overlap matchings $f,g$ is a $\delta$-interleaving if and only if the pair $\barcToDgm(f)$, $\barcToDgm(g)$ is a $\delta$-interleaving in $\MatchingCat^\RCat$.

\paragraph{Interleavings and Smallness of Kernels}

It is easily checked that for  $\A$ a Puppe-exact category, a $\delta$-interleaving morphism $f: M \to N(\delta)$ has $2\delta$-trivial kernel and cokernel. 
The converse is not true in general; one can easily construct a counterexample in the case that $\A$ is the category of persistence modules.  However, the converse holds in the two cases studied in this paper:

\begin{proposition}
\label{prop:interleavingTriviality}
In both the categories $\kvect^\RCat$ and $\Barc$, two objects $M,N$ are $\delta$-interleaved if and only if there exists a morphism $f: M \to N(\delta)$ with $2\delta$-trivial kernel and cokernel.
\end{proposition}

The statement of \cref{prop:interleavingTriviality} for $\kvect^\RCat$ first appeared as \cite[Corollary~6.6]{bauer2015induced}.

\begin{proof}
The result for $\Barc$ follows easily from \cref{Prop:Triviality_Kers_Cokers_of_Matching_Diagram}.

To prove the result for $\kvect^\RCat$, we apply both the induced matching and converse algebraic stability theorems: If $f: M \to N(\delta)$ is a morphism with $2\delta$-trivial kernel and cokernel, then by \cref{Thm:IMT_New_Formulation}, $\IndMat(f)$ has the same property.  Hence $\IndMat(f)$ is a $\delta$-interleaving morphism.  The converse direction of \cref{Thm:AST_New_Formulation} (whose easy proof we give below) then tells us that $M$ and $N$ are $\delta$-interleaved. 
\end{proof}

\subsection{Algebraic Stability}\label{Sec:Algebraic_Stability}

\paragraph{Bottleneck Distance}

For $I\subset \R$ an interval and $\delta\geq 0$, let the interval $U_\delta(I)$ be given by \[U_\delta(I):=\{t\in \R\mid \exists\, s\in I\textup{ with }|s-t|\leq \delta\}.\]  
We define a \emph{$\delta$-matching} between barcodes $\mathcal C$ and $\mathcal D$ to be a (not necessarily overlap) matching $\sigma:\C \matching \D$ with the following two properties:
\begin{enumerate}[(i)]
\item $\sigma$ matches each interval in $\C \cup \D$ that is not $2\delta$-trivial,
\item if 
$\sigma(I)=J$, then
$I\subset U_\delta(J)$ and 
$J\subset U_\delta(I)$.
\end{enumerate}
We define the \emph{bottleneck distance $d_B$} by taking
\[d_B(\C,\D):=\inf\, \{\delta\geq 0\mid\exists\textup{ a $\delta$-matching between $\C$ and $\D$}\}.\]
\paragraph{Interleaving distance equals bottleneck distance on barcodes}
For $\D$ any barcode, let $r_\delta:\D(\delta)\to \D$ be the obvious bijection.

\begin{proposition}\label{Prop:dI=db}
An overlap matching of barcodes $f:\C\matching\D(\delta)$ is a $\delta$-interleaving morphism if and only if $r_\delta\circ f$ is a $\delta$-matching.  In particular, for any barcodes $\C$ and $\D$, \[d_I(\C,\D)=d_B(\C,\D).\]
\end{proposition}

\begin{proof}
According to \cref{prop:interleavingTriviality}, an overlap matching $f:\C\matching \D(\delta)$ is a $\delta$-interleaving morphism if and only if $f$ has $2\delta$-trivial kernel and cokernel.  In addition, it is easy to check that an overlap matching $f:\C\matching\D(\delta)$ has $2\delta$-trivial kernel and cokernel if and only if $r_\delta\circ f$ is a $\delta$-matching. 
\end{proof}

\begin{figure}
\centering

\begin{tikzpicture}[scale=0.7,baseline=(bottom),x=1cm,y=1cm]

\fill[myblue,opacity=.1] (6,6.75) -- (6.75,6.75) -- (8,8) -- (8,8.75) -- (6,8.75);
\draw[thick,dash pattern=on 0pt off 2pt,line cap=round,thick] (7,7.75) -- (7,7);
\fill[myblue] (7,7.75) circle [radius=3pt];

\fill[myred,opacity=.1] (1.2,7.3) rectangle (3.2,9.3);
\fill[myblue,opacity=.1] (3,8.75) rectangle (1,6.75);
\draw[dash pattern=on 0pt off 2pt,line cap=round,thick] (2,7.75) -- (2,7.3) -- (1.2,7.3) -- (2.2,8.3);
\fill[myblue] (2,7.75) circle [radius=3pt];
\fill[mygreen] (2,7.3) circle [radius=3pt];
\fill[myyellow] (1.2,7.3) circle [radius=3pt];
\fill[myred] (2.2,8.3) circle [radius=3pt];

\fill[myred,opacity=.1] (2.5,4) -- (4,4) -- (4.5,4.5) -- (4.5,6) -- (2.5,6);
\fill[myblue,opacity=.1] (4.1,6.1) rectangle (2.1,4.1);
\draw[dash pattern=on 0pt off 2pt,line cap=round,thick] (3,5.1) -- (3,4) -- (2.5,4) -- (3.5,5);
\fill[myblue] (3,5.1) circle [radius=3pt];
\fill[mygreen] (3,4) circle [radius=3pt];
\fill[myyellow] (2.5,4) circle [radius=3pt];
\fill[myred] (3.5,5) circle [radius=3pt];

\fill[myred,opacity=.1] (1,1.4) -- (1.4,1.4) -- (3,3) -- (3,3.4) -- (1,3.4);
\draw[dash pattern=on 0pt off 2pt,line cap=round,thick] (1.4,1.4) -- (1,1.4) -- (2,2.4);
\fill[myyellow] (1,1.4) circle [radius=3pt];
\fill[myred] (2,2.4) circle [radius=3pt];

\draw[->] (0,0) -- (0,9);
\draw[->] (0,0) -- (9,0);
\draw[dashed,lightgray] (0,0) -- (9,9);

\node[myblue] at (9,4) [right] {$\B(M)$};

\node[mygreen] at (9,3) [right] {$\B(\im f)$};

\node[myyellow] (bottom) at (9,2) [right] {$\B(N(\delta))$};

\node[myred] (bottom) at (9,1) [right] {$\B(N)$};

\end{tikzpicture}
\caption{Illustration of the proof of algebraic stability via induced matchings.  For a $\delta$-interleaving morphism of persistence modules $f: M \to N$, the barcodes of $M$, $\im f$, $N(\delta)$, and $N$ are shown as \emph{persistence diagrams}, which are multisets of points in the plane whose coordinates correspond to the left and right endpoints of the intervals. The dotted lines in the figure depict the induced matching $\B(M) \to \B(\im f) \to \B(N(\delta)) \to \B(N)$.  The shaded box around each point $p\in\B(M)\cup\B(N)$ indicates the set of points to which $p$ can match in a $\delta$-matching.}
\label{fig:AST_example}
\end{figure}

We are now deduce the algebraic stability theorem as a corollary of the induced matching theorem.  \cref{fig:AST_example} illustrates the barcode matching underlying the argument.

\begin{proof}[Proof of Algebraic Stability (\cref{Thm:AST_New_Formulation})]
The forward direction follows almost immediately from the induced matching theorem: If there exists a $\delta$-interleaving morphism $f:M\to N(\delta)$, then $f$ has $2\delta$-trivial kernel and cokernel.  By \cref{Thm:IMT_New_Formulation}, the same is true for $\IndMat(f):\B(M)\to \B(N(\delta))$.  Since $\B(N(\delta))=\B(N)(\delta)$, \cref{prop:interleavingTriviality} tells us that $\B(M)$ and $\B(N)$ are $\delta$-interleaved in $\Barc$.

The proof of converse algebraic stability is nearly trivial: Given a $\delta$-interleaving \[f:\B(M)\to \B(N)(\delta),\quad g:\B(N)\to \B(M)(\delta),\] $\zeta(f)$ and $\zeta(g)$ form a $\delta$-interleaving in $\kvect^\RCat$; here  $\zeta$ is the functor defined in \cref{Sec:Converse_IMT}.  By \cref{Prop:Conv_IMT}\,(i) then, $M$ and $N$ are $\delta$-interleaved.  
\end{proof}

\section{Constructing barcodes and induced matchings directly in $\MatchingCat^\RCat$}
\label{sec:rank_barcode}
In this section, we consider the construction of barcodes of persistence modules and induced matchings directly in the category of matching diagrams $\MatchingCat^\RCat$.  
Our barcode constructions come in two dual (canonically isomorphic) variants, which are readily extended to functors on epis and monos, respectively.  
These functors are equivalent to the induced matchings for epis and monos described in \cref{Sec: Induced_Matcings_Monos_Epis} and lead naturally to an alternate proof of \cref{prop:SubmoduleStructureCat}.  %

Let $M$ be a \pfd persistence module.  We now construct a matching diagram $\barcEpi(M)$ equivalent to $\B(M)$ in a way that depends only on the ranks of the internal maps of $M$.
While our construction does not require an interval decomposition of $M$, the intuition is best conveyed by assuming initially that we have this.

Order the intervals in $\R$ lexicographically, first by increasing lower bound, then (for intervals with the same lower bound) by decreasing upper bound, as shown in \cref{fig:lex_barcode}.
\begin{figure}
\centering
\begin{tikzpicture}[scale=1,baseline=(current bounding box.north),x=4ex,y=-2.5ex]
\draw[gray] (2.5,-1) -- (2.5,6);
\node[gray] (s) at (2.5,6.5) {$t$};
\draw[gray] (4,-1) -- (4,6);
\node[gray] (t) at (4,6.5) {$u$};
\draw[ultra thick,myblue] (0,0) -- (5,0);
\draw[ultra thick,myblue] (0,1) -- (3.25,1);
\draw[ultra thick,myblue] (2,2) -- (6,2);
\draw[ultra thick,myblue] (2,3) -- (4.5,3);
\draw[ultra thick,myblue] (2,4) -- (3.5,4);
\draw[ultra thick,myblue] (3,5) -- (7,5);
\node[circle, myblue, thick, draw, fill=white, scale=.8, inner sep=1, minimum size=10pt] at (2.5,0) {1};
\node[circle, gray, draw, fill=white, scale=.8, inner sep=1, minimum size=10pt] at (2.5,1) {2};
\node[circle, myblue, thick, draw, fill=white, scale=.8, inner sep=1, minimum size=10pt] at (2.5,2) {3};
\node[circle, myblue, thick, draw, fill=white, scale=.8, inner sep=1, minimum size=10pt] at (2.5,3) {4};
\node[circle, gray, draw, fill=white, scale=.8, inner sep=1, minimum size=10pt] at (2.5,4) {5};
\node[circle, myblue, thick, draw, fill=white, scale=.8, inner sep=1, minimum size=10pt] at (4,0) {1};
\node[circle, myblue, thick, draw, fill=white, scale=.8, inner sep=1, minimum size=10pt] at (4,2) {2};
\node[circle, myblue, thick, draw, fill=white, scale=.8, inner sep=1, minimum size=10pt] at (4,3) {3};
\node[circle, gray, draw, fill=white, scale=.8, inner sep=1, minimum size=10pt] at (4,5) {4};
\end{tikzpicture}
\caption{Example illustrating the matching diagram $\barcEpi(M)$ on top of the barcode $\B(M)$,
with intervals ordered lexicographically by increasing lower bound and deceasing upper bound.
Each set $\barcEpi(M)_t \subset \mathbb N$ is identified by an order-preserving bijection with the intervals of $\B(M)$ containing $t$.
The example shows a matching of cardinality $3$, corresponding to $\rank M_{t,u} = 3$.
\label{fig:lex_barcode}
}
\end{figure}
For example, with respect to this order, we have $[0,3]<[1,2]<(1,3]<(1,2]$.  Now at each index $t$, enumerate the intervals of $\B(M)$ containing~$t$ in that order.  This defines a canonical bijection $g_t \colon \barcToDgm(\B(M))_t \to \barcEpi(M)_t$ between the set $\barcToDgm(\B(M))_t$, consisting of the intervals in $\B(M)$ containing $t$,
and the set \[\barcEpi(M)_t := \{1 ,2, \dots , \dim M_t \}.\]
For any two indices $t \leq u$, let
\[\barcEpi(M)_{t,u} := g_u^{-1} \circ \barcToDgm(\B(M))_{t,u} \circ g_t.\]
Thus, $\barcEpi(M)_{t,u}$ is the matching between the sets $\barcEpi(M)_t$ and $\barcEpi(M)_u$ such that under the bijections $g_t$ and $g_u$, matched pairs correspond to intervals of $\B(M)$ containing both $t$ and $u$; 
see \cref{fig:lex_barcode}.
By construction, the matchings $(\barcEpi(M)_{t,u})_{t\leq u\in \R}$ form  a functor $\barcEpi(M):\RCat\to \MatchingCat$, and bijections $(g_t)_{t\in \R}$ form a natural isomorphism of matching diagrams $g:\barcToDgm(\B(M))\to \barcEpi(M)$.
A dual construction, denoted by $\barcMono$, is obtained by ordering the intervals in $\R$ lexicographically by decreasing upper bound, then by increasing lower bound. 
We summarize:

\begin{proposition}\label{Two_Matching_Dgms_Naturally_Isomorphic}
The matching diagrams $\barcEpi(M)$ and $\barcMono(M)$ are naturally isomorphic to $\barcToDgm(\B(M))$.
\end{proposition}

In a similar spirit, we can also map the induced matchings for epimorphisms and monomorphisms to equivalent morphisms of matching diagrams.
Consider an epimorphism $f : M \twoheadrightarrow N$ and $\barcEpi(M), \barcEpi(N)$ as above.
Now, for any $t \in \R$, let $g_t \colon \barcToDgm(\B(M))_t \to \barcEpi(M)_t$ and $h_t \colon \barcToDgm(\B(N))_t \to \barcEpi(N)_t$ be the canonical bijections described above.
Using the induced matching $\IndMat(f) \colon \B(M) \to \B(N)$ from \cref{Sec: Induced_Matcings_Monos_Epis}, define
\[\barcEpi(f)_{t} := h_u^{-1} \circ \barcToDgm(\IndMat(f))_{t} \circ g_t.\]
See \cref{fig:lex_barcode_epi} for an example.
For a monomorphism $f$, we define $\barcMono(f)_{t}$ in an analogous way.
\begin{figure}
\centering
\begin{tikzpicture}[scale=1,baseline=(current bounding box.north),x=4ex,y=-5ex]
\draw[gray] (2.5,-0.5) -- (2.5,6);
\node[gray] (s) at (2.5,6.5) {$t$};
\fill[mylightgray] (0,0) rectangle (3.5,0.45);
\fill[mylightgray] (0,1) rectangle (1.75,1.45);
\fill[mylightgray] (2,2) rectangle (5,2.45);
\fill[mylightgray] (2,3) rectangle (3,3.45);
\fill[mylightgray] (3,5) rectangle (5,5.45);
\draw[ultra thick,myblue] (0,0) -- (5,0);
\draw[ultra thick,myblue] (0,1) -- (3.25,1);
\draw[ultra thick,myblue] (2,2) -- (6,2);
\draw[ultra thick,myblue] (2,3) -- (4.5,3);
\draw[ultra thick,myblue] (2,4) -- (3.5,4);
\draw[ultra thick,myblue] (3,5) -- (7,5);
\draw[ultra thick,myyellow] (0,0.45) -- (3.5,0.45);
\draw[ultra thick,myyellow] (0,1.45) -- (1.75,1.45);
\draw[ultra thick,myyellow] (2,2.45) -- (5,2.45);
\draw[ultra thick,myyellow] (2,3.45) -- (3,3.45);
\draw[ultra thick,myyellow] (3,5.45) -- (5,5.45);

\node[circle, myblue, thick, draw, fill=white, scale=.8, inner sep=1, minimum size=10pt] at (2.5,0) {1};
\node[circle, gray, draw, fill=white, scale=.8, inner sep=1, minimum size=10pt] at (2.5,1) {2};
\node[circle, myblue, thick, draw, fill=white, scale=.8, inner sep=1, minimum size=10pt] at (2.5,2) {3};
\node[circle, myblue, thick, draw, fill=white, scale=.8, inner sep=1, minimum size=10pt] at (2.5,3) {4};
\node[circle, gray, draw, fill=white, scale=.8, inner sep=1, minimum size=10pt] at (2.5,4) {5};
\node[circle, myyellow, thick, draw, fill=white, scale=.8, inner sep=1, minimum size=10pt] at (2.5,0.45) {1};
\node[circle, myyellow, thick, draw, fill=white, scale=.8, inner sep=1, minimum size=10pt] at (2.5,2.45) {2};
\node[circle, myyellow, thick, draw, fill=white, scale=.8, inner sep=1, minimum size=10pt] at (2.5,3.45) {3};
\end{tikzpicture}
\caption{Example illustrating the induced epimorphism of 
matching diagrams $\barcEpi(f) : \barcEpi(M) \twoheadrightarrow \barcEpi(N)$ and the corresponding overlap matching of barcodes $\B(M) \twoheadrightarrow \B(N)$.
\label{fig:lex_barcode_epi}
}
\end{figure}
Similarly to the above, we obtain:
\begin{proposition}\label{Induced_Matchings_Naturally_Isomorphic}
For an epimorphism $f : M \twoheadrightarrow N$,
the morphism $\barcEpi(f)$ is naturally isomorphic to $\barcToDgm(\IndMat(f))$.
Similarly, for a monomorphism $f$,
the morphism $\barcMono(f)$ is naturally isomorphic to $\barcToDgm(\IndMat(f))$.
\end{proposition}

Next, we describe the matching $\barcEpi(M)_{t,u}$ directly in terms of the internal maps of $M$, avoiding the explicit use of the barcode $\B(M)$.

\begin{proposition}\label{Def:MatchingCatBarc}
For $M: {\cat R} \to \cat{Vect}$ be a \pfd persistence module and 
$t \leq u \in \R$, we have
\begin{align*}
\barcEpi(M)_{t,u} &= \left\{(i,j) \in 
\mathbb N^2
\mid j \leq \rank M_{t,u} , \, i = j + \max\,\{\rank M_{s,t} - \rank M_{s,u} \mid s < t, \rank M_{s,u} < j \} \right\}  , 
\\
\barcMono(M)_{t,u} &= \left\{(i,j) \in 
\mathbb N^2
\mid i \leq \rank M_{t,u} , \, j = i + \max\,\{\rank M_{u,v} - \rank M_{t,v} \mid v > u, \rank M_{t,v} < i \} \right\}  ,
\end{align*}
(where the maximum over an empty set is taken to be $0$).
\end{proposition}
\begin{proof}

\goodbreak

We first observe that the image of $\barcEpi(M)_{t,u}$ is precisely the set \(
\{j \mid j \leq \rank M_{t,u}\}.
\)
In order to determine for a given matched number $j \in \im \barcEpi(M)_{t,u}$ the corresponding number $i \in \barcEpi(M)_t$ to which is it matched, we further observe that the difference $i-j$ is precisely the number of intervals of $\B(M)$ that 
\begin{enumerate}[(a)]
\item are born before the $j^{\mathrm{th}}$ interval of $\B(M)$ containing $u$  (in the lexicographic order), and
\item die after $t$ and before $u$.
\end{enumerate}
Letting $I$ be the $j^{\mathrm{th}}$~interval of $\B(M)$ containing $u$, 
the set of lower bounds of $I$ in $\R$ (i.e., the set of values $s \in \R$ satisfying  $s < r$ for all $r \in I$) is
\[
\{s < u \mid \rank M_{s,u} < j\}.
\]
If $j\in  \im \barcEpi(M)$, then $t\in I$, and clearly all such lower bounds $s$ satisfy $s<t$.  Hence, the number of intervals that are born before~$I$ and die after $t$ is \[\max_s \,\{\rank M_{s,t} \mid s < t,\, \rank M_{s,u} < j \},\] and similarly, the number of intervals born before $I$ which die after $u$ is \[\max_s \,\{\rank M_{s,u} \mid s < t,\, \rank M_{s,u} < j \}.\]
It follows that 
\begin{align*}
i - j &= \max_s \,\{\rank M_{s,t} \mid s < t,\, \rank M_{s,u} < j \}- \max_s \,\{\rank M_{s,u} \mid s < t,\, \rank M_{s,u} < j \}\\
      &=\max_s\,\{\rank M_{s,t} - \rank M_{s,u} \mid s < t,\, \rank M_{s,u} < j \}.
\end{align*}
The result follows.
\end{proof}

In an analogous way, we also obtain formulas for $\barcEpi(f)_{t}$ and $\barcMono(f)_{t}$ in terms of the morphism $f$ and the internal maps, which can be proven in a similar way.

\begin{proposition}
\label{prop:episFunctor}
\mbox{}
\begin{enumerate}[(i)]
\item Let $f: M \twoheadrightarrow N$ be an epimorphism of \pfd persistence modules.  For $t\in \R$, we have 
\[
\barcEpi(f)_{t} = \left\{(i,j) \in \mathbb N^2 \mid j \leq \dim N_{t} , \, i = j + \max\,\{\rank M_{s,t} - \rank N_{s,t} \mid s < t, \rank N_{s,t} < j \} \right\}.
\]
\item Let $f: M \twoheadrightarrow N$ be a monomorphism of \pfd persistence modules.  For $t\in \R$, we have 
\[
\barcMono(f)_{t} = \left\{(i,j) \in \mathbb N^2 \mid i \leq \dim M_{t} , \, j = i + \max\,\{\rank N_{s,t} - \rank M_{s,t} \mid s < t, \rank M_{s,t} < i \} \right\}.
\]
\end{enumerate}
\end{proposition}

Note that these formulas rely only on the existence of an epimorphism or monomorphism; the right hand sides depend only on the ranks of the internal maps of $M$ and $N$, not on the morphism $f$.

We have seen that the formulas of \cref{Def:MatchingCatBarc} define functors $\barcEpi(M), \barcMono(M) \colon \RCat\to \MatchingCat$, and it is clear from \cref{Two_Matching_Dgms_Naturally_Isomorphic} that each of these functors encodes $\rank M_{u,t}$ for all $u\leq t\in \R$.  In the sprit of constructing and studying $\barcEpi(M)$ and $\barcMono(M)$ in a way that is independent of the structure theorem, we next give elementary proofs of these facts, proceeding directly from the description of $\barcEpi(M)$ and $\barcMono(M)$ in terms of ranks.

\begin{proposition}
\label{prop:epiMonoMatchingDgms}
Let $M: {\cat R} \to \cat{Vect}$ be a \pfd persistence module. 
\begin{enumerate}[(i)]
\item For all $t\leq u\in \R$, the relations $\barcEpi(M)_{t,u}$ and $\barcMono(M)_{t,u} $ of \cref{Def:MatchingCatBarc} are both order-preserving matchings.  In particular, for $t < u$, \[\card \barcEpi(M)_{t,u} = \card \barcMono(M)_{t,u} = \rank M_{t,u} .\]
\item The sets and matchings of \cref{Def:MatchingCatBarc} are functorial, i.e., they define functors 
\begin{align*}
\barcEpi(M),\barcMono(M)&: \cat R \to \cat{Mch}.
\end{align*}
\end{enumerate}
\end{proposition}

\begin{proof}

We prove the results for $\barcEpi(M)$ only, the proof for $\barcMono(M)$ being completely analogous.
Let 
$(i,j), (m,n) \in \barcEpi(M)_{t,u}$.
Clearly $j=n$ implies $i=m$.
Moreover, if $j < n$, then from
\begin{align*}
i = j + \max\,\{\rank M_{s,t} - \rank M_{s,u} \mid s < t, \rank M_{s,u} < j \} , \\
m = n + \max\,\{\rank M_{s,t} - \rank M_{s,u} \mid s < t, \rank M_{s,u} < n \} .
\end{align*}
we obtain $i < m$.
Thus $\barcEpi(M)_{t,u}$ is an order-preserving matching.

In order to show functoriality of $\barcEpi(M)$, we first establish that for all $s < t \leq u$, we have
$\rank M_{s,t} < i$ if and only if $\rank M_{s,u} < j$.  
To see this, note that for all $s < u$ with $\rank M_{s,u} < j$, we have $s < t$ and
\[i \geq j + (\rank M_{s,t} - \rank M_{s,u}) > \rank M_{s,t}.\] 
Conversely, for all $r < u$ with $\rank M_{r,u} \geq j$, we have $s < r$ for all $s < t$ with $\rank M_{s,u} < j$,
which in turn by elementary linear algebra yields
\[\rank M_{s,t} - \rank M_{s,u} 
= \dim (\im M_{s,t} \cap \ker M_{t,u})
\leq \dim (\im M_{r,t} \cap \ker M_{t,u})
=
\rank M_{r,t} - \rank M_{r,u}\]
and thus
\[i = j + \max\{\rank M_{s,t} - \rank M_{s,u} \mid s < t, \rank M_{s,u} < j \} \leq j + (\rank M_{r,t} - \rank M_{r,u}) \leq \rank M_{r,t} .\]
We conclude that $\rank M_{s,t} < i $ if and only if $\rank M_{s,u} < j$.

It remains to show that $(i,k) \in \barcEpi(M)_{t,v}$ if and only if $(i,j) \in \barcEpi(M)_{t,u}$
and
$(j,k) \in \barcEpi(M)_{u,v}$ for some $j \in \barcEpi(M)_{u}$. 
First let
$(i,j) \in \barcEpi(M)_{t,u}$
and
$(j,k) \in \barcEpi(M)_{u,v}$.
By the above we have
$\rank M_{s,u} < j$
if and only if
$\rank M_{s,v} < k$,
and so substituting
\[j = k + \max\,\{\rank M_{s,u} - \rank M_{s,v} \mid s < t, \rank M_{s,v} < k \}\]
gives
\begin{align*}
i = k &+ \max\, \{\rank M_{s,u} - \rank M_{s,v} \mid s < t, \rank M_{s,v} < k \} \\&+ \max\,\{\rank M_{s,t} - \rank M_{s,u} \mid s < t, \rank M_{s,u} < j \} 
\\
= k &+ \max\,\{\rank M_{s,t} - \rank M_{s,v} \mid s < t, \rank M_{s,v} < k \}  ,
\end{align*}
which is equivalent to $(i,k) \in \barcEpi(M)_{t,v}$.
Conversely, given $(i,k) \in \barcEpi(M)_{t,v}$, the above equation for $j$ yields $(i,j) \in \barcEpi(M)_{t,u}$
and
$(j,k) \in \barcEpi(M)_{u,v}$.
We conclude that $\barcEpi(M)$ is a functor ${\cat R} \to \MatchingCat$.
\end{proof}

Similarly, we can also show directly from the description of \cref {prop:episFunctor} that $\barcEpi(f)$ and $\barcMono(f)$ are natural transformations, turning $\barcEpi$ and $\barcMono$ into functors.
We omit the proof, which is essentially the same as the proof of \cref{prop:epiMonoMatchingDgms}.

\goodbreak

\begin{proposition}
\label{prop:episToEpis}
\mbox{}
\begin{enumerate}[(i)]
\item Let $f: M \twoheadrightarrow N$ be an epimorphism.  Then, for all $t$, $\barcEpi(f)_{t}$ is an order-preserving epimorphism in $\MatchingCat$.  Moreover, these matchings are natural, so they define an epimorphism \[\barcEpi(f):\barcEpi(M) \twoheadrightarrow \barcEpi(N)\]
in a functorial way, i.e., $\barcEpi(g \circ f) = \barcEpi(g) \circ \barcEpi(f)$ for any epimorphism $g: N \twoheadrightarrow O$.
\item Let $f: M \hookrightarrow N$ be a monomorphism.  Then, for all $t$, $\barcMono(f)_{t}$ is an order-preserving monomorphism in $\MatchingCat$.  Moreover, these matchings are natural, so they define a monomorphism \[\barcMono(f):\barcMono(M) \hookrightarrow \barcMono(N).\]
in a functorial way, i.e., $\barcMono(g \circ f) = \barcMono(g) \circ \barcMono(f)$ for any monomorphism $g: N \hookrightarrow O$.
\end{enumerate}
\end{proposition}

\begin{remark}
As an aside, we note that the formulas of \cref{prop:episFunctor,prop:epiMonoMatchingDgms} extend to any \emph{q-tame} persistence module $M$ (i.e., one for which $\rank(M_{s,t})<\infty$ whenever $s<t$), even though the usual structure theorem for \pfd persistence modules does not extend to the q-tame setting \cite{chazal2012structure}.  However, since in this setting \cref{Two_Matching_Dgms_Naturally_Isomorphic} does not apply, it is not guaranteed that the resulting matching diagrams $\barcEpi(M)$ and $\barcMono(M)$ are isomorphic.
\end{remark}

\begin{remark}[Matchings induced by arbitrary morphisms]
While $\barcEpi(M)$ and $\barcMono(M)$ are typically not equal, we have seen above that there is a distinguished isomorphism from each of these matching diagrams to $\barcToDgm(\B(M))$.  This in turn gives us a distinguished isomorphism \[\zeta_M:\barcEpi(M)\to \barcMono(M).\]   Using this, we can define the matching $\barcEpi(M)\to \barcEpi(N)$ induced by a morphism $f:M\to N$ of \pfd persistence modules as the composition of matching diagrams
\[
\begin{tikzcd}
\barcEpi(M) \ar[r,"\barcEpi(q)"] & \barcEpi(\im f) \ar[r,"\zeta_{\im f}"] & \barcMono(\im f) \ar[r,"\barcMono(i)"] &  \barcMono(N)  \ar[r,"\zeta_N^{-1}"] & \barcEpi(N),
\end{tikzcd}
\]
where $f=i\circ q$ is the epi-mono factorization of $f$.  By construction, this is equivalent to the induced matching $\IndMat(f)$ in $\Barc$.  A definition of the matching $\barcMono(M)\to \barcMono(N)$ induced by $f$ can be given in a similar way.  

Because $\zeta_{\im f}$ and $\zeta_{N}$ are defined in terms of barcodes, our definition of the matching $\barcEpi(M)\to \barcEpi(N)$ induced by $f$ is defined in terms of barcodes as well.  This is at odds with the goal of giving a barcode-free construction of induced matchings directly in $\MatchingCat^\RCat$, as we have done when restricting attention to monos or epis $f$.  However, we do not see a simple way to define the matching $\barcEpi(M)\to \barcEpi(N)$ induced by an arbitrary morphism $f$ without appealing to the connection with barcodes.  This suggests to us that to define matchings induced by arbitrary morphisms, it is more natural to work in the category $\Barc$, as we have done elsewhere in this paper, than to work in $\MatchingCat^\RCat$.

\end{remark}

\section{Discussion}
In this paper, we have established some basic facts about the category $\Barc\cong \MatchingCat^\RCat$ of barcodes and used these observations to give simple new formulations of the induced matching and algebraic stability  theorems.  We have seen that the new formulations lead to variant of the proof of the induced matching theorem which emphasizes the preservation of categorical structure.

In fact, our definition of the category $\Barc$ extends to barcodes indexed over arbitrary posets, as defined in \cite{Botnan2018Algebraic}, and many of the ideas presented here extend either to arbitrary posets or to $\RCat^n$-indexed barcodes for any $n$.  In particular, \cref{Prop:dI=db} extends to $\RCat^n$-indexed barcodes, and this provides alternative language for expressing generalized algebraic stability results appearing in \cite{Botnan2018Algebraic,bjerkevik2016stability}.  While it remains to be seen what role the categorical viewpoint on barcodes might play in the further development of TDA theory, we hope that it 
might offer some perspective on how algebraic stability ought to generalize to other settings.

As already mentioned, our new formulations of the algebraic stability and induced matching theorems make clear that both results can be interpreted as the preservation of some categorical structure as we pass from $\kvect^\RCat$ to $\Barc$.  Can more of interest be said about how the passage from persistence modules to barcodes preserves categorical structure?  
We wonder whether our results can be understood as part of a larger story about how homological algebra in the Abelian category $\kvect^\RCat$ relates to homological algebra in $\Barc$.

\section*{Acknowledgements}
This research has been supported by the DFG Collaborative Research Center SFB/TRR 109 ``Discretization in Geometry and Dynamics'', NIH grant T32MH065214, and an award from the J.\@ Insley Blair Pyne Fund.

\bibliographystyle{abbrvnaturl}
{\small \bibliography{DaD_Refs} }

\end{document}